\documentclass{amsart}

\newtheorem{theorem}[equation]{Theorem}
\newtheorem{lemma}[equation]{Lemma}

\newtheorem{proposition}[equation]{Proposition}

\numberwithin{equation}{section}

\usepackage{amscd,amsmath,amssymb,verbatim}

\begin{document}

\title{A note on the integrality of mirror maps}
\author{Alan Adolphson}
\address{Department of Mathematics\\
Oklahoma State University\\
Stillwater, Oklahoma 74078}
\email{alan.adolphson@okstate.edu}
\author{Steven Sperber}
\address{School of Mathematics\\
University of Minnesota\\
Minneapolis, Minnesota 55455}
\email{sperber@math.umn.edu}
\date{\today}
\begin{abstract}
We give a class of examples of $A$-hypergeometric systems that display integrality of mirror maps.  Specifically, these systems have solutions $F(\lambda_1,\dots,\lambda_N)=1$ and $\log\lambda^l + G(\lambda_1,\dots,\lambda_N)$ (for certain $l\in{\mathbb Z}^N$) such that $\exp G(\lambda)$ has integral coefficients.  The proof requires only some elementary congruences.
\end{abstract}
\maketitle

\section{Introduction}

There is a substantial literature on the integrality of mirror maps.  Typically one starts with a differential equation having series solutions $F(\lambda)$ and $F(\lambda)\log\lambda + G(\lambda)$ and looks for conditions that imply that $\exp(G(\lambda)/F(\lambda))$ has integral coefficients (or $p$-integral coefficients for some prime $p$).  We refer the reader to Delaygue, Rivoal, and Roques \cite{DRR} and Krattenthaler and Rivoal \cite{KR} for some foundational results and further references.  Many approaches to this problem are based on $p$-adic congruences due to Dwork \cite{D,D2} and usually involve substantial computations.  Our aim here is to give a simpler example of this phenomenon.

Let $A=\{{\bf a}_j\}_{j=1}^N$ be elements of ${\mathbb Z}^n$ and write ${\bf a}_j = (a_{1j},\dots,a_{nj})$ for each~$j$.  We describe the associated $A$-hypergeometric system with parameter $\beta=(\beta_1,\dots,\beta_n)\in{\mathbb C}^n$.  It consists of the box operators
\begin{equation}
\Box_l = \prod_{l_j>0}\bigg(\frac{\partial}{\partial\lambda_j}\bigg)^{l_j} - \prod_{l_j<0}\bigg(\frac{\partial}{\partial\lambda_j}\bigg)^{-l_j}
\end{equation}
for each $l\in L$, where $L$ is the lattice of relations on the set $A$,
\[ L = \bigg\{ l=(l_1,\dots,l_N)\in{\mathbb Z}^N \bigg| \sum_{j=1}^N l_j{\bf a}_j = {\bf 0}\bigg\}, \]
and the Euler operators
\begin{equation}
Z_i=\sum_{j=1}^N a_{ij}\lambda_j\frac{\partial}{\partial\lambda_j}-\beta_i
\end{equation}
for $i=1\dots,n$.  

We assume always that there is a homogeneous linear form $h=\sum_{i=1}^n h_ix_i$ such that $h({\bf a}_j) = 1$ for all $j$.
This implies in particular that the system is regular holonomic.  We note that for $l=(l_1,\dots,l_N)\in L$ we have
\begin{equation}
\sum_{j=1}^N l_j = \sum_{j=1}^N l_jh({\bf a}_j) = h\bigg(\sum_{j=1}^N l_j{\bf a}_j\bigg) = h({\bf 0}) = 0.
\end{equation}
This equation implies that if $l\in L\setminus\{{\bf 0}\}$, then $l_j>0$ for some $j$ and $l_j<0$ for some $j$.  This implies that the constant function $F(\lambda_1,\dots,\lambda_N)=1$ satisfies the box operators (1.1).  

In this paper, we are interested in the $A$-hypergeometric systems with parameter $\beta = {\bf 0}$.  It is clear that the constant function $F(\lambda)=1$ satisfies the Euler operators (1.2) when $\beta={\bf 0}$.  We record these facts as a proposition.
\begin{proposition}
The constant function $F(\lambda)=1$ satisfies the $A$-hypergeometric system $(1.1)$, $(1.2)$ for the parameter $\beta={\bf 0}$.
\end{proposition}

\section{Logarithmic solutions}

We construct the logarithmic solutions associated to the solution $F(\lambda)=1$.  
For $k=1,\dots,N$, set
\[ L_k = \{ l=(l_1,\dots,l_N)\in L\mid \text{$l_k\leq 0$ and $l_j\geq 0$ for $j\neq k$}\} \]
and define
\[ G_k(\lambda) = \sum_{\substack{l=(l_1,\dots,l_N)\in L_k\\ l_k<0}} (-1)^{-l_k-1} \frac{(-l_k-1)!}{\displaystyle \prod_{\substack{j=1\\ j\neq k}}^N l_j!}\lambda^l. \]
We first encountered these series because of their connection with the Hasse-Witt matrix of a projective hypersurface over a finite field \cite[Section~3]{AS1}.

\begin{lemma}
For $k=1,\dots,N$, the series $G_k(\lambda)$ satisfies the Euler operators $(1.2)$ with parameter $\beta = {\bf 0}$ and $\log\lambda_k + G_k(\lambda)$ satisfies the box operators $(1.1)$.
\end{lemma}

\begin{proof}
Since $\lambda_j\frac{\partial\lambda^l}{\partial\lambda_j}=l_j\lambda^l$, one has
\begin{multline*}
 (Z_1(G_k),\dots,Z_n(G_k)) =  \\
 \sum_{\substack{l=(l_1,\dots,l_N)\in L_k\\ l_k<0}} (-1)^{-l_k-1} \frac{(-l_k-1)!}{\displaystyle \prod_{\substack{j=1\\ j\neq k}}^N l_j!}\bigg(\sum_{j=1}^N l_j a_{1j},\dots,\sum_{j=1}^N l_ja_{nj}\bigg)\lambda^l ={\bf 0} 
 \end{multline*}
since $l\in L$.  This verifies that $G_k$ satisfies the Euler operators with $\beta={\bf 0}$.

To check the other assertion, it is useful to recall the construction of the $G_k$ using generating series \cite{AS}.  Define
\[ \xi(t) = \sum_{i=0}^\infty \frac{t^i}{i!}\bigg(\log t-\bigg(1 + \frac{1}{2} + \cdots + \frac{1}{i}\bigg)\bigg) + \sum_{i=-1}^{-\infty} (-1)^{i-1} (-i-1)!t^i. \]
For a vector $u=(u_1,\dots,u_N)\in{\mathbb C}^N$ we write $T^u = T_1^{u_1}\cdots T_N^{u_N}$.  We set
\begin{equation}
\Phi(\lambda_k,T^{{\bf a}_k}) =   \sum_{i=0}^\infty \frac{\lambda_k^i}{i!}\bigg(\log \lambda_k-\bigg(1 + \frac{1}{2} + \cdots + \frac{1}{i}\bigg)\bigg) T^{i{\bf a}_k}+ \sum_{i=-1}^{-\infty} (-1)^{i-1} (-i-1)!\lambda_k^iT^{i{\bf a}_k}. 
\end{equation}
Note that $d\xi/dt = \xi$, i.~e., the derivative of the $i$-th term is the $(i-1)$-st term for all $i\in{\mathbb Z}$.  This implies that
\begin{equation}
\frac{\partial}{\partial\lambda_i}\Phi(\lambda_k,T^{{\bf a}_k}) = \begin{cases} T^{{\bf a}_k}\Phi(\lambda_k,T^{{\bf a}_k}) & \text{if $i=k$,} \\
0 & \text{if $i\neq k$.} \end{cases}
\end{equation}

Define $\Psi(\lambda,T) = \Phi(\lambda_k,T^{{\bf a}_k})\prod_{\substack{j=1\\ j\neq k}}^N \exp(\lambda_jT^{{\bf a}_j})$.  
We have
\begin{equation}
\frac{\partial}{\partial\lambda_i}\exp(\lambda_jT^{{\bf a}_j}) = \begin{cases} T^{{\bf a}_j}\exp(\lambda_jT^{{\bf a}_j}) & \text{if $i=j$,} \\
0 & \text{if $i\neq j$.} \end{cases}
\end{equation}
Let $\tilde{l}=(\tilde{l}_1,\dots,\tilde{l}_N)\in L$.  It follows from (2.3) and (2.4) that
\begin{equation}
\Box_{\tilde{l}}(\Psi(\lambda,T)) = (T^{\sum_{\tilde{l}_i>0}\tilde{l}_i{\bf a}_i}-T^{-\sum_{\tilde{l}_i<0}\tilde{l}_i{\bf a}_i})\Psi(\lambda,T) = 0
\end{equation}
since $\sum_{\tilde{l}_i>0}\tilde{l}_i{\bf a}_i = -\sum_{\tilde{l}_i<0}\tilde{l}_i{\bf a}_i$.  

This implies that, for $u\in{\mathbb Z}^N$, the coefficient of $T^u$ in $\Psi(\lambda,T)$ satisfies the box operators.
Note that $\log\lambda_k + G_k(\lambda)$ is the coefficient of $T^{\bf 0}$ in $\Psi(\lambda,T)$, so it in particular satisfies the box operators. 
\end{proof}

The function $\log\lambda_k + G_k(\lambda)$ is not a solution of the $A$-hypergeometric system with parameter $\beta = {\bf 0}$ because $\log\lambda_k$ fails to satisfy the Euler operators with parameter $\beta = {\bf 0}$.  One has to take appropriate linear combinations to get such solutions.
\begin{proposition}
For all $\tilde{l} = (\tilde{l}_1,\dots,\tilde{l}_N)\in L$, the expression
\begin{equation}
 \sum_{k=1}^N \tilde{l}_k(\log\lambda_k + G_k(\lambda)) = \log\lambda^{\tilde{l}} +\sum_{k=1}^N \tilde{l}_kG_k(\lambda) 
 \end{equation}
satisfies the $A$-hypergeometric system with parameter $\beta = {\bf 0}$.
\end{proposition}

\begin{proof}
It is straightforward to check that $\log\lambda^{\tilde{l}}$ satisfies the Euler operators with parameter $\beta = {\bf 0}$.  The proposition then follows from Lemma~2.1.
\end{proof}

The solutions of Proposition 2.6 are, {\it \`a priori}, formal solutions.  In Section 5 below we show that they belong to a Nilsson ring, hence are actual solutions.  

\section{Some elementary congruences}

In this section we record the elementary congruences needed to prove that $\exp G_k(\lambda)$ has integral coefficients.  These congruences may be well known, but we are unaware of a convenient reference so we include their simple proofs.

Let $e_1,\dots,e_N$ be nonnegative integers and set $e=\sum_{i=1}^N e_i$.  We consider the multinomial coefficients
\[ \binom{e}{e_1,\dots,e_N} = \frac{e!}{e_1!\cdots e_N!}. \]
Let $p^a$ be the highest power of the prime $p$ dividing $e$ and let $p^b$ be the highest power of $p$ dividing all $e_i$.  We have $a\geq b$ since $e$ is the sum of the $e_i$.  

\begin{proposition}  
The multinomial coefficient $\binom{e}{e_1,\dots,e_N}$ is divisible by $p^{a-b}$.
\end{proposition}

\begin{proof}
If all $e_i=0$ the result is trivial (taking $\infty-\infty = 0$).  So suppose $e_1>0$.  Then
\[ \binom{e}{e_1,\dots,e_N} = \frac{e}{e_1}\binom{e-1}{e_1-1,e_2,\dots,e_N}. \]
The fraction $e/e_1$ is divisible by $p^{a-b}$ and $\binom{e-1}{e_1-1,e_2,\dots,e_N}$ is an integer.
\end{proof}

\begin{proposition}
The difference
\begin{equation}
\binom{pe}{pe_1,\dots,pe_N}-\binom{e}{e_1,\dots,e_N}
\end{equation}
is divisible by $p^{a+1}$.
\end{proposition}

\begin{proof}
The factors of $(pe)!$ divisible by $p$ are
\[ (pe)(p(e-1))\cdots (p\cdot 1) = p^e e!. \]
It follows that we can write
\[ \binom{pe}{pe_1,\dots,pe_N} = \frac{p^e e!}{p^{e_1+\cdots+e_N}e_1!\cdots e_N!}B = \binom{e}{e_1,\dots,e_N}B, \]
where the rational number $B$ is the quotient of all factors in $\binom{pe}{pe_1,\dots,pe_N}$ that are prime to $p$.  The difference (3.3) may thus be written
\[ \binom{e}{e_1,\dots,e_N}(B-1). \]
By Proposition 3.1, to prove Proposition 3.2 it thus suffices to prove that
\begin{equation}
B-1\in p^{b+1}({\mathbb Q}\cap{\mathbb Z}_p).
\end{equation}

We first factor $B$ as a product of rational numbers $B_1,\dots,B_N$ defined as follows.  Set $e_0=0$ for indexing purposes.  The denominator of $B_i$ is the product of those factors of $(pe_i)!$ that are prime to $p$:
\begin{equation}
\text{denominator of $B_i$} = \prod_{\substack{j=0\\ p\not\:|j}}^{pe_i-1}(pe_i-j).
\end{equation}
The numerator of $B_i$ is the product of those factors between $p(e-e_1-\cdots-e_{i-1})$ and $p(e-e_1-\cdots-e_i)$ that are prime to $p$:
\begin{equation}
\text{numerator of $B_i$} = \prod_{\substack{j=0\\ p\not\:|j}}^{pe_i-1} (p(e-e_1-\cdots-e_{i-1})-j).
\end{equation}
We now have $B=B_1\cdots B_N$.  Furthermore, the ratio of a term of (3.6) to the corresponding term of (3.5) is
\begin{equation}
\frac{p(e-e_1-\cdots-e_{i-1})-j}{pe_i-j}.
\end{equation}
Since $e$ and all $e_i$ are divisible by $p^b$ and $j$ is prime to $p$, the ratio (3.7) lies in $1+p^{b+1}({\mathbb Q}\cap{\mathbb Z}_p)$.  It follows that all $B_i$ and therefore $B$ itself lie in $1+p^{b+1}({\mathbb Q}\cap{\mathbb Z}_p)$.  This establishes (3.4).  
\end{proof}

\section{Integrality}

Note that the series $G_k(\lambda)$ has exponents $l=(l_1,\dots,l_N)$ in the orthant where $l_k\leq 0$ and $l_j\geq 0$ for $j\neq k$.  Additionally we always have $l_k<0$, so the series $\exp G_k(\lambda)$ is well-defined as a power series with exponents in that same orthant.
In this section we prove the main theorem:
\begin{theorem}
For $k=1,\dots,N$, the series $\exp G_k(\lambda)$ has integral coefficients.
\end{theorem}

\begin{proof}
By the multivariable version of the Dieudonn\'e-Dwork Lemma (Krattenthaler and Rivoal \cite[Lemmas~1 and~2]{KR}) the series 
$\exp G_k(\lambda)$ will have $p$-integral coefficients for a prime $p$ if and only if
\begin{equation}
pG_k(\lambda) - G_k(\lambda^p)\in p({\mathbb Q}\cap{\mathbb Z}_p)[[\lambda_1,\dots,\lambda_k^{-1},\dots,\lambda_N]].
\end{equation}
We prove this condition is satisfied for all primes $p$.

From the definition of $G_k(\lambda)$, this requires the following two types of congruences to be satisfied.  If $(l_1,\dots,l_N)\in L_k$ and not all of $l_1,\dots,l_N$ are divisible by $p$, we need
\begin{equation}
 p\frac{(-l_k-1)!}{\displaystyle\prod_{\substack{j=1\\ j\neq k}}^N l_j!}\in p({\mathbb Q}\cap{\mathbb Z}_p).
 \end{equation}
 And for all $(l_1,\dots,l_N)\in L_k$ we need
 \begin{equation}
 p(-1)^{-pl_k-1}\frac{(-pl_k-1)!}{\displaystyle\prod_{\substack{j=1\\ j\neq k}}^N (pl_j)!}-(-1)^{-l_k-1} \frac{(-l_k-1)!}{\displaystyle\prod_{\substack{j=1\\ j\neq k}}^N l_j!}\in p({\mathbb Q}\cap{\mathbb Z}_p).
 \end{equation}
 
 From (1.3) we have $-l_k = \sum_{\substack{j=1\\ j\neq k}}^N l_j$, so (4.3) is equivalent to the assertion that
 \begin{equation}
 -\frac{1}{l_k} \binom{-l_k}{l_1,\dots,\hat{l}_k,\dots,l_N}\in {\mathbb Q}\cap{\mathbb Z}_p.
 \end{equation}
We now apply Proposition 3.1.  Take $p^a$ to be the highest power of $p$ dividing $l_k$ and $p^b$ to be the highest power of $p$ dividing all $l_j$ for $j\neq k$.  We are assuming that $b=0$, so $\binom{-l_k}{l_1,\dots,\hat{l}_k,\dots,l_N}$ is divisible by $p^a$.  This proves (4.5).

Now consider (4.4).  We have $(-1)^{-pl_k-1} = (-1)^{-l_k-1}$ except when $p=2$ and $l_k$ is odd.  Putting aside this exceptional case for the moment, (4.4) is equivalent to 
\begin{equation}
-\frac{1}{l_k} \bigg(\binom{-pl_k}{pl_1,\dots,\widehat{pl}_k,\dots,pl_N} - \binom{-l_k}{l_1,\dots,\hat{l}_k,\dots,l_N}\bigg) \in p({\mathbb Q}\cap{\mathbb Z}_p).
\end{equation}
By Proposition 3.2 the difference inside the parentheses is divisible by $p^{a+1}$, which establishes (4.6) if we are not in the exceptional case.  In the exceptional case, we must show that
\begin{equation}
-\frac{1}{l_k} \bigg(\binom{-2l_k}{2l_1,\dots,\widehat{2l}_k,\dots,2l_N} + \binom{-l_k}{l_1,\dots,\hat{l}_k,\dots,l_N}\bigg) \in 2({\mathbb Q}\cap{\mathbb Z}_2)
\end{equation}
when $l_k$ is odd.  For $l_k$ odd we have $-\frac{1}{l_k} \binom{-l_k}{l_1,\dots,\hat{l}_k,\dots,l_N}\in {\mathbb Q}\cap{\mathbb Z}_2$, so (4.7) is equivalent to
\begin{equation}
-\frac{1}{l_k} \bigg(\binom{-2l_k}{2l_1,\dots,\widehat{2l}_k,\dots,2l_N} - \binom{-l_k}{l_1,\dots,\hat{l}_k,\dots,l_N}\bigg) \in 2({\mathbb Q}\cap{\mathbb Z}_2).
\end{equation}
But this is the case $p=2$ of (4.6), which has already been proved.
\end{proof}

\section{Logarithmic solutions, II}

As noted in Proposition 2.6, formal solutions of the $A$-hypergeometric system with parameter $\beta = {\bf 0}$ are of the form (2.7).  Although the series $\exp G_k(\lambda)$ are well-defined and have integral coefficients, we need to show that the same is true for $\exp\big(\sum_{k=1}^N \tilde{l}_kG_k(\lambda)\big)$ for all $\tilde{l} = (\tilde{l}_1,\dots,\tilde{l}_N)\in L$.  This will be the case if all $G_k(\lambda)$ lie in a common ring of series.

Let $C$ be the real cone generated by $\bigcup_{k=1}^N L_k$, so the exponents of all monomials in all $G_k(\lambda)$ lie in $C$.  Let $S$ be the ${\mathbb C}$-module of all series in the $\lambda^u$, $u\in {\mathbb Z}^N\cap C$, with ${\mathbb C}$-coefficients.  Then $G_k(\lambda)\in S$ for all $k$.  We verify that $S$ is a ring by checking that the cone $C$ has a vertex at the origin.  The argument follows the lines of the proof of \cite[Proposition 2.9]{AS1}.

\begin{proposition}
The cone $C$ has a vertex at the origin.
\end{proposition}

\begin{proof}
If the assertion of the proposition were false, a subspace of ${\mathbb R}^N$ would be contained in $C$.  This would imply that the origin is an interior point of the convex hull of a subset of $\bigcup_{k=1}^N L_k$:
\[ {\bf 0} = \sum_{i=1}^M c_i\ell_i, \]
where the $c_i$ are positive rational numbers summing to 1 and the $\ell_i$ lie in $\bigcup_{k=1}^N L_k$.  We can multiply this equation by a nonnegative integer and assume the $c_i$ are positive integers.  And since each set $L_k$ is closed under nonnegative integral linear combinations we may assume 
\begin{equation}
{\bf 0} = \sum_{k=1}^N l^{(k)},
\end{equation}
where $l^{(k)}=(l^{(k)}_1,\dots,l^{(k)}_N)\in L_k$ for $k=1,\dots,N$.
To prove that the origin is a vertex of $C$, we show that (5.2) holds if and only if $l^{(k)}={\bf 0}$ for all $k$.

The ``if'' assertion is trivial, so we show that (5.2) implies $l^{(k)}={\bf 0}$ for $k=1,\dots,N$.  Since $l^{(k)}\in L_k$ we have
\begin{equation}
\sum_{j=1}^N l^{(k)}_j{\bf a}_j = 0\quad\text{for each $k=1,\dots,N$}.
\end{equation}
If $l^{(k)}\neq {\bf 0}$, then by the definition of $L_k$ and (1.3) we have $l^{(k)}_k<0$ and $l^{(k)}_j>0$ for at least one $j$, $j\neq k$.  We can then solve Equation (5.3) for ${\bf a}_k$:
\begin{equation}
{\bf a}_k = \sum_{\substack{j=1\\ j\neq k}}^N \bigg(-\frac{l^{(k)}_j}{l^{(k)}_k}\bigg){\bf a}_j.
\end{equation}
The coefficients on the right-hand side of Equation (5.4) are nonnegative rational numbers that sum to 1 by (1.3), so (5.4) implies the following statement:
\begin{lemma}
If $l^{(k)}\neq 0$, then ${\bf a}_k$ is an interior point of the convex hull of the set $\{{\bf a}_j\mid l^{(k)}_j>0\}$.  
\end{lemma}

Let
\[ D = \{{\bf a}_j\mid \text{$l^{(k)}_j\neq 0$ for some $k\in\{1,\dots,N\}$}\}. \]
Then $l^{(k)}={\bf 0}$ for all $k$ if and only if $D=\emptyset$.  We prove that $l^{(k)}={\bf 0}$ for all $k$ by assuming that $D\neq \emptyset$ and deriving a contradiction.

If $D\neq\emptyset$ we may choose $j_0\in\{1,\dots,N\}$ for which ${\bf a}_{j_0}$ is a vertex of the convex hull of $D$.  Lemma~5.5 implies that $l^{(j_0)}={\bf 0}$.  In particular, if $l^{(k)}\neq{\bf 0}$, then $k\neq j_0$.  But by the definition of $L_k$, the only coordinate of $l^{(k)}$ that can be negative is $l^{(k)}_k$, all other coordinates must be nonegative.  So $l^{(k)}\neq{\bf 0}$ implies that $l^{(k)}_{j_0}\geq 0$.  Furthermore, since ${\bf a}_{j_0}\in D$, we must have $l^{(k)}_{j_0}>0$ for some $k$.  But this implies that $\sum_{k=1}^N l^{(k)}_{j_0}>0$, contradicting (5.2).  
\end{proof}

\end{document}